\documentclass[12pt,twoside,a4paper]{article}

\usepackage{graphicx}
\usepackage{amsmath,amsfonts,amsthm}
\usepackage{array}
\usepackage{dcolumn}
\usepackage{psfrag}
\usepackage{latexsym}
\usepackage{float}
\usepackage{array}
\usepackage{dcolumn}

\usepackage{epsfig}
\usepackage{enumerate,color}
\usepackage{psfrag}
\usepackage{graphics}
\usepackage{amsmath,amsfonts,amsthm}
\usepackage[cp850]{inputenc}

\usepackage{url,hyperref}

\usepackage{epstopdf}
\usepackage{color}
\usepackage[T1]{fontenc}
\makeatletter
\newcommand\footnoteref[1]{\protected@xdef\@thefnmark{\ref{#1}}\@footnotemark}
\makeatother

\newtheorem{lemma}{Lemma}
\newtheorem{theorem}{Theorem}

\def\N{\mathbb{N}}
\def\C{\mathbb{C}}
\def\R{\mathbb{R}}
\def\aa{\mathbf{a}}
\def\bb{\mathbf{b}}

\title{Uniformly convergent expansions for the generalized hypergeometric functions of the Bessel and Kummer types}
\author{Jos\'e L.L\'opez$^*$\,\,,\,\,Pedro J.Pagola$^*$\,\,,\,\,Dmitrii B.Karp$^\dag$
	\\
	\\
	\textsf{\textit{$*$ Dpto. de Ingenier\'{\i}a Matem\'{a}tica e Inform\'{a}tica,}}\\
    \textsf{\textit{Universidad P\'{u}blica de Navarra, Spain}}\\
	\\
	\textsf{\textit{$\dag$ Far Eastern Federal University, Vladivostok, Russia}}\\
	\textsf{\textit{and Institute of Applied Mathematics FEBRAS}}
\\[8pt]
	\textit{email: jl.lopez@unavarra.es, pedro.pagola@unavarra.es, dimkrp@gmail.com}}
\date{}
\begin{document}
	
	\normalsize \maketitle
	
\begin{abstract}
We derive a convergent expansion of the generalized hypergeometric function ${}_{p-1}F_p$ in terms of the Bessel functions ${}_{0}F_1$ that holds uniformly with respect to the argument in any  horizontal strip of the complex plane. We further obtain a convergent expansion of the generalized hypergeometric function ${}_{p}F_p$ in terms of the confluent hypergeometric functions ${}_{1}F_1$ that holds uniformly in any right half-plane. For both functions, we make a further step and give convergent expansions in terms of trigonometric, exponential and rational functions that hold uniformly in the same domains. For all four expansions we present explicit error bounds. The accuracy of the approximations is illustrated with some numerical experiments.
\vspace{1cm}
		
\noindent \textsf{2010 AMS \textit{Mathematics Subject Classification:} 33C20; 41A58; 41A80.} \\\\
\noindent  \textsf{Keywords \& Phrases:} generalized hypergeometric function; Bessel function; Kummer function; convergent expansions; uniform expansions.
		\\\\
\end{abstract}

\section{Introduction}

A variety of expansions (convergent or asymptotic) of the special functions of mathematical physics can be found in the literature.
These expansions have the important property of being given in terms of elementary functions: mostly, positive or negative powers of a certain variable $z$ and, sometimes, other elementary functions. However, very often, these expansions are not simultaneously valid for small and large values of $\vert z\vert$. Thus, it would be interesting to derive new convergent expansions in terms of elementary functions that hold uniformly in $z$ in a large region of the complex plane containing both small and large values of $\vert z\vert$.

In \cite{blanca}, \cite{confluent} and \cite{Lopez}, the authors derived new uniform convergent expansions of the incomplete gamma functions, the Bessel functions and the confluent hypergeometric functions, respectively, in terms of elementary functions. The starting point of the technique used in \cite{blanca}, \cite{confluent} and \cite{Lopez} is an appropriate integral representation of these functions. The key idea is the use of the Taylor expansion of a certain factor of the integrand that is independent of the variable $z$, at an appropriate point of the integration interval, and subsequent interchange of sum and integral. The independence of that factor of $z$ translates into  uniform convergence of the resulting expansion in a large region of the complex $z-$plane. The expansions given in \cite{blanca}, \cite{confluent} and \cite{Lopez} are accompanied by error bounds and numerical experiments showing the accuracy of the approximations.

In this work, we continue that line of investigation by considering the generalized hypergeometric functions ${}_{p-1}F_p(\aa;\bb;z)$ and ${}_{p}F_p(\aa;\bb;z)$. We view them as functions of the complex variable $z$, and derive new convergent expansions uniformly valid in an unbounded region of the complex $z-$plane that contains the point $z=0$. The generalized hypergeometric function (GHF) ${}_{q}F_p(\aa;\bb;z)$ is defined by means of the hypergeometric series as
(see \cite[Section~2.1]{AAR}, \cite[Section~5.1]{LukeBook}, \cite[Chapter~12]{BealsWong} or \cite[eq. (16.2.1)]{NIST})
\begin{equation}\label{qfp}
\left.{}_{q}F_p\left(\begin{matrix}\aa\\\bb\end{matrix}\:\right\vert z\right)=\sum_{k=0}^{\infty}\frac{\left(a_{1}\right)_{k}\cdots\left(a_{q%
	}\right)_{k}}{\left(b_{1}\right)_{k}\cdots\left(b_{p}\right)_{k}}\frac{z^{k}}{%
	k!},
\end{equation}
where $\aa:=(a_1,a_2,\cdots,a_q)$ and $\bb:=(b_1,b_2,\cdots,b_p)$, $-b_j\notin\mathbb{N}\cup\{0\}$, are parameter vectors and $(a)_n=\Gamma(a+n)/\Gamma(a)$ is the Pochhammer's symbol. In general, ${}_{q}F_p(\aa;\bb;z)$ does not exist when some $b_k=0,-1,-2,\ldots$. Series  \eqref{qfp} converges $\forall z\in\C$ if $q\leq p$ and inside the unit disk if $q=p+1$. In the latter case, the generalized hypergeometric function ${}_{p+1}F_p(\aa;\bb;z)$ is defined outside the unit disk by analytic continuation to the cut plane  $\mathbb{C}\setminus[1,\infty)$ and the branch defined in this way in the sector $|\arg(1-z)|<\pi$, is called the principal branch (or principal value) of ${}_{p+1}F_p(\aa;\bb;z)$.

In the remaining part of this paper we only consider $q=p-1$ or $q=p$ and $\Re a_k>0$, $k=1,2,3,...,q$.
In the case $q=p-1$ we assume that $\Re[(b_1+b_2+...+b_p)-(a_1+a_2+...+a_{p-1})]>1/2$. In this case our starting point is the integral representation of ${}_{p-1}F_p(\textbf{a};\textbf{b};z)$ originally derived by Kiryakova \cite[Chapter~4]{KiryakovaBook} and further discussed in \cite[eq.(12)]{dimitri} and \cite[eq.(28)]{KLJAT2017} which, when combined with the shifting property (\ref{eq:Gtimespower}), takes the form:
\begin{equation}\label{ini}
\left.{}_{p-1}F_p\left(\begin{matrix}\aa\\ \bb\end{matrix}\:\right\vert -{z^2\over 4}\right)=
{2\Gamma(\bb)\over\sqrt{\pi}\,\Gamma(\aa)}\int_0^1\cos(zt)\,
G_{p,p}^{p,0}\left(t^2\:\vline~\begin{matrix}\bb-1/2\\\aa-1/2,0\end{matrix}\right)\!dt,
\hskip 1cm z\in\mathbb{C},
\end{equation}
where $G_{p,p}^{p,0}$ is a particular case of Meijer's $G$ function defined and further explained in (\ref{eq:G-defined}) below.
If $q=p$ we assume that $\Re[(b_1+b_2+...+b_p)-(a_1+a_2+...+a_p)]>0$. In this case our starting point is the integral representation of ${}_pF_p(\textbf{a};\textbf{b};z)$ derived in \cite[Chapter~4]{KiryakovaBook} and  further  discussed in \cite[eq.(11)]{dimitri},
\begin{equation}\label{inidos}
\left.{}_{p}F_p\left(\begin{matrix}\aa\\ \bb\end{matrix}\:\right\vert -z\right)=
{\Gamma(\bb)\over\,\Gamma(\aa)}\int_0^1e^{-zt}\,
G_{p,p}^{p,0}\left(t\:\vline~\begin{matrix}\bb-1\\ \aa-1\end{matrix}\right)\!dt,
\hskip 1cm z\in\mathbb{C}.
\end{equation}
If the above restrictions on parameters are violated the results of this paper can still be applied by employing the decomposition \cite[eq.(31)]{KLJAT2017}
\begin{equation}\label{eq:pFq-decompose}
{}_{q}F_p\left.\left(\begin{matrix}\aa\\\bb\end{matrix}\:\right\vert z\right)=
\sum_{k=0}^{n-1}\frac{(\aa)_k}{(\bb)_kk!}z^k
+\frac{(\aa)_nz^{n}}{(\bb)_nn!}{}_{q+1}F_{p+1}\left.\left(\begin{matrix}\aa+n,1\\\bb+n,n+1\end{matrix}\right\vert z\right).
\end{equation}
Indeed, we can always choose $n$ large enough to satisfy $\Re(\aa+n)>0$ and $\Re[(b_1+b_2+...+b_p+np+n+1)-(a_1+a_2+...+a_q+nq+1)]>0$.

The power series expansion \eqref{qfp} may be obtained from \eqref{ini} or \eqref{inidos} by replacing the factor $\cos(zt)$ or $e^{-zt}$ by its Taylor series at the origin, interchanging series and integral and using the following formula for the moments of the $G_{p,p}^{p,0}$ function \cite[eq.(16)]{KLJAT2017},
$$
\int_0^1 t^m G_{p,p}^{p,0}\left(t\:\vline~\begin{matrix}\bb\\\aa\end{matrix}\right)\!dt=\frac{\Gamma(\aa+m+1)}{\Gamma(\bb+m+1)}, \hskip 1cm m\in\N_0,
$$
The Taylor expansions for $\cos(zt)$ and $e^{-zt}$ converge for $t\in[0,1]$, but the convergence is not uniform in $|z|$. Therefore, expansion \eqref{qfp} is convergent, but not uniformly in $|z|$ as the remainder is unbounded for large $|z|$.

The asymptotic expansions of ${}_{p-1}F_p(\aa;\bb;z)$ and ${}_pF_p(\aa;\bb;z)$ for large $\vert z\vert$ can be found in \cite[Sec. 16.11]{NIST}. They are given in terms of formal series expansions in inverse powers of $z$ and are asymptotic for large $\vert z\vert$, but the remainders are unbounded for small $|z|$ and then, the expansions are not uniform in $|z|$.

As an illustration of the uniform approximations that we are going to obtain in this paper (see Theorems~\ref{th:GHFBesselBessel}-\ref{th:GHFKummerelementary} below), we derive, for example, the following one:
\begin{equation}\label{ejemplo}
{}_1F_2\left(3;\frac{7}{2},5;-\frac{z^2}{4}\right)\simeq\;\frac{720z(8z^4+105z^2-1890)}{z^{11}}\cos z
+\frac{720(z^6-15z^4-735z^2+1890)}{z^{11}}\sin z,
\end{equation}
that approximates the left hand side in any horizontal strip of the complex $z-$plane. Note that the limit of the right hand side of \eqref{ejemplo} as $z\to0$ is finite and equals $208/231$. Figure 1 illustrates the accuracy and the uniform character of the above approximation for real $z$.

\begin{figure}[h!]
  \centering
   \includegraphics[width=6cm]{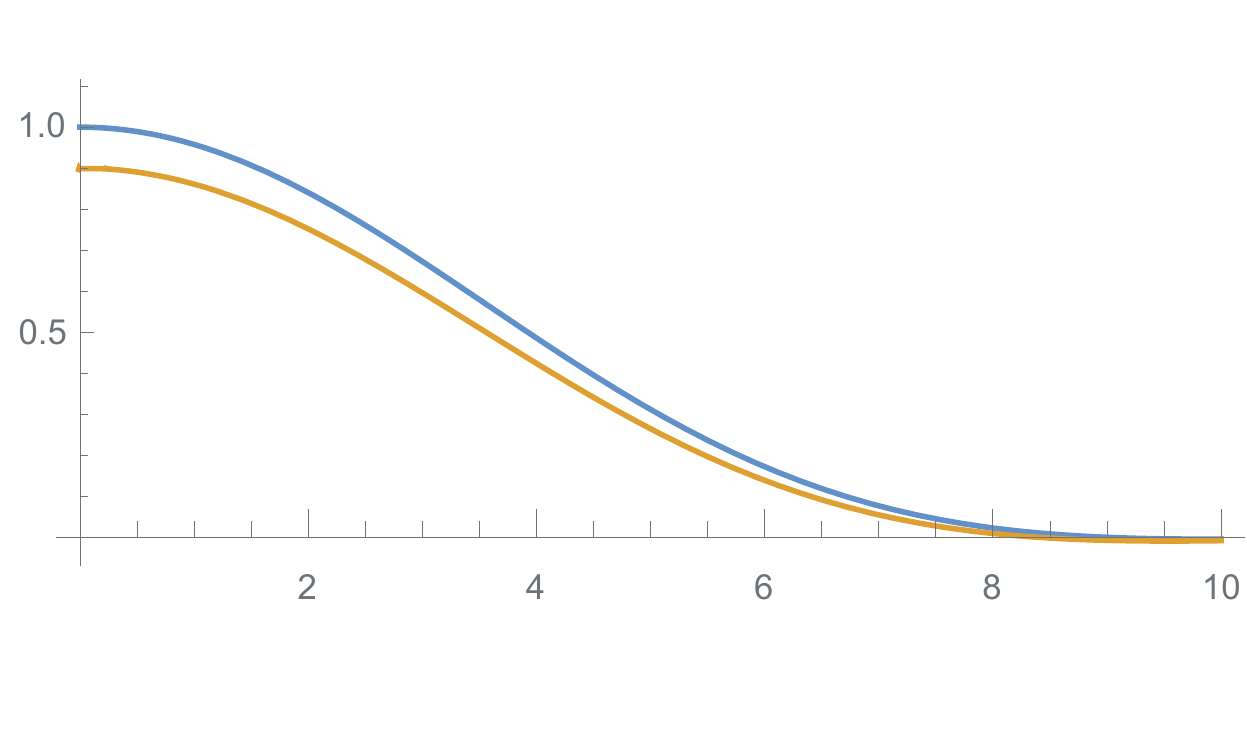}
\hskip 2cm
  \includegraphics[width=6cm]{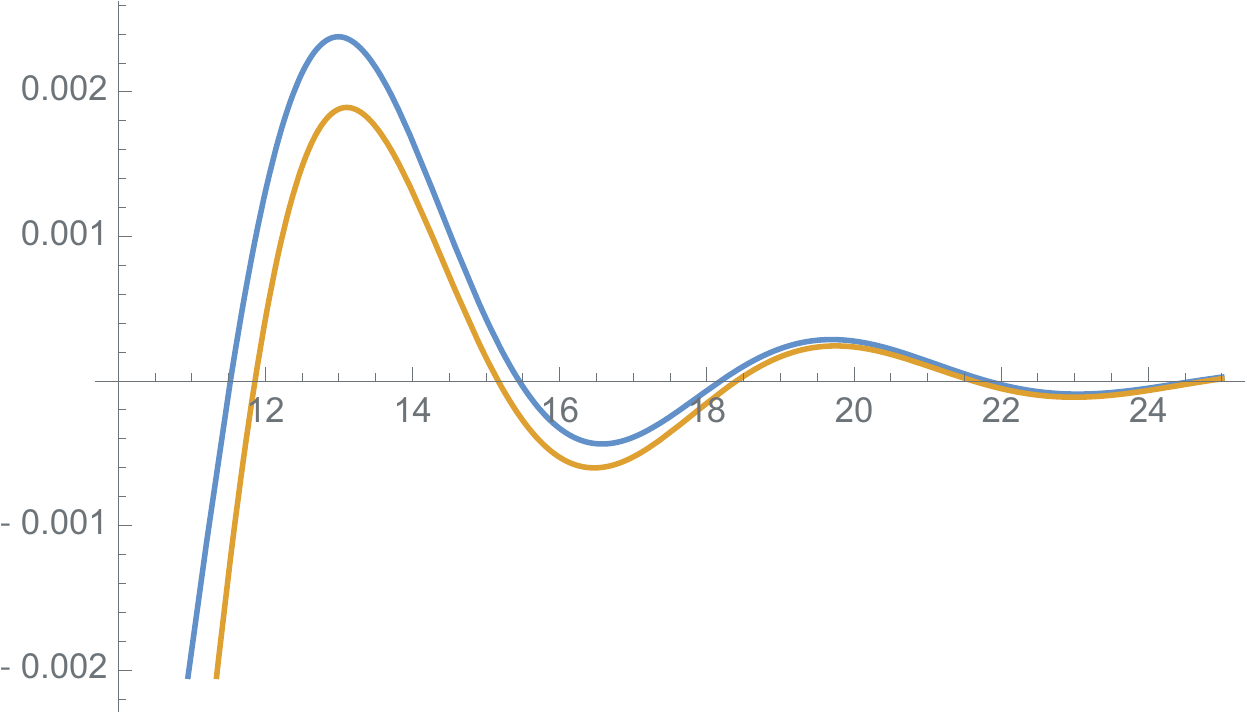}
  \caption{\small Plot of the left (blue) and right (yellow) hand side of \eqref{ejemplo} in two intervals of the real $z-$axis. }\label{figure1}
\end{figure}
\noindent

In order to derive uniformly convergent expansions of ${}_{p-1}F_p(\aa;\bb;z)$ and ${}_pF_p(\aa;\bb;z)$, we apply the technique proposed in \cite{blanca}, \cite{confluent} and \cite{Lopez}: we consider the Taylor expansion of the factor $G_{p,p}^{p,0}(\aa;\bb;t)$ at $t=1$ in \eqref{ini} and \eqref{inidos}. This Taylor expansion is convergent for any $t$ in the interval of integration and, obviously, it is independent of $z$. After the interchange of the series and the integral, this independence translates into a remainder that may be bounded uniformly with respect to $z$ in a large unbounded region of the complex $z-$plane that contains the point $z=0$ and that we specify in Theorems~\ref{th:GHFBesselBessel}-\ref{th:GHFKummerelementary} below.

This paper is organized as follows. In the preliminary Section 2, some properties of the Meijer-N{\o}rlund function $G^{p,0}_{p,p}$ and N{\o}rlund's coefficients needed for later computations are presented and some notation introduced.  In  Section 3 we consider the integral representation \eqref{ini} for ${}_{p-1}F_p(\aa;\bb;z)$. In Section 4 we consider the integral representation \eqref{inidos} for ${}_pF_p(\aa;\bb;z)$.   In these two sections we first derive expansions in terms of the Bessel and the confluent hypergeometric functions, respectively. We may consider these expansion as ''natural'', as the Bessel function ${}_{0}F_1(\aa;\bb;z)$ and the confluent hypergeometric function ${}_1F_1(\aa;\bb;z)$ are the first functions of the respective $p-$hierachies. Next, using the known expansions of the Bessel and the confluent hypergeometric functions in terms of elementary functions from \cite{confluent} and \cite{Lopez}, we proceed to derive, for both ${}_{p-1}F_p(\aa;\bb;z)$ and ${}_pF_p(\aa;\bb;z)$, a second expansion in terms of elementary functions.  Throughout the paper we use the principal argument $\arg(z)\in(-\pi,\pi]$.

\section{Preliminaries on the Meijer-N{\o}rlund function and N{\o}rlund's coefficients}

We will use the standard notation $\mathbb{N}$, $\mathbb{Z}$ and $\mathbb{C}$ for the sets of natural, integer and complex numbers, respectively;  $\mathbb{N}_0=\mathbb{N}\cup\{0\}$. The size of a vector $\aa:=(a_1,a_2,...,a_p)$ is typically obvious from the subscript of the corresponding hypergeometric function.  Throughout the paper, we will use the shorthand notation for  products and sums:
$$
\Gamma(\aa):=\Gamma(a_1)\Gamma(a_2)\cdots\Gamma(a_p),~~(\aa)_n:=(a_1)_n(a_2)_n\cdots(a_p)_n,~~\textbf{a}+\mu:=(a_1+\mu,a_2+\mu,\dots,a_p+\mu);
$$
inequalities like $\Re(\aa)>0$ and properties like $-\aa\notin\mathbb{N}_0$ will be understood element-wise.  The symbol $\aa_{[k]}$ stands for the vector $\aa$ with omitted $k$-th component. Given two complex vectors $\aa\in\mathbb{C}^{q}$, $\bb\in\mathbb{C}^{p}$, we also define
\begin{equation}\label{eq:psi-defined}
\psi(\aa,\bb):=\sum_{j=1}^{p}b_j-\sum_{j=1}^{q}a_j.
\end{equation}

We will need the basic properties of a particular case of Meijer's $G$ function $G^{p,0}_{p,p}$ studied in detail by N{\o}rlund in \cite{Norlund} using a different notation and without mentioning Meijer's previous work. In \cite{KL2018IJAM} we suggested the denomination ''Meijer-N{\o}rlund function'' for this function defined by the Mellin-Barnes integral of the form
\begin{equation}\label{eq:G-defined}
G^{p,0}_{p,p}\!\left(\!z~\vline\begin{array}{l}\bb\\\aa\end{array}\!\!\right)\!\!:=
\\
\frac{1}{2\pi{i}}
\int\limits_{\mathcal{L}}\!\!\frac{\Gamma(\aa\!+\!s)}{\Gamma(\bb+\!s)}z^{-s}ds, \hskip 2cm z\in\mathbb{C}.
\end{equation}
We omit the details regarding the choice of the contour ${\mathcal{L}}$ as the definition of (the general case of) Meijer's $G$ function can be found in standard text- and reference- books \cite[section~5.2]{LukeBook}, \cite[16.17]{NIST}, \cite[8.2]{PBM3} and \cite[Chapter~12]{BealsWong}.  See also our papers \cite{KLJAT2017,KL2018IJAM,KPSIGMA}.  The following shifting property is straightforward from the definition (\ref{eq:G-defined}), but nevertheless it is very useful (see \cite[8.2.2.15]{PBM3} or \cite[Sec. 16.19, eq. (16.19.2)]{NIST}):
\begin{equation}\label{eq:Gtimespower}
z^{\alpha}G^{p,0}_{p,p}\!\left(\!z~\vline\begin{array}{l}\bb\\\aa\end{array}\!\!\right)=G^{p,0}_{p,p}\!\left(\!z~\vline\begin{array}{l}\bb+\alpha\\\aa+\alpha\end{array}\!\!\right), \hskip 2cm \alpha\in\mathbb{C}.
\end{equation}
Given two complex vectors $\aa\in\mathbb{C}^{p-1}$, $\bb\in\mathbb{C}^{p}$ and $N\in\mathbb{N}$, N{\o}rlund's coefficients $g_n(\aa;\bb)$ are defined  via the generating function \cite[eq.(1.33)]{Norlund}, \cite[eq.(11)]{KLJAT2017} which we present in a split form for further reference:

\begin{equation}\label{eq:Norlund}
G^{p,0}_{p,p}\!\left(\!1-z~\vline\begin{array}{l}\bb\\\aa,0\end{array}\!\!\right)=\frac{z^{\psi(\aa;\bb)-1}}{\Gamma(\psi(\aa,\bb))}
\sum\limits_{n=0}^{N-1}\frac{g_n(\aa;\bb)}{(\psi(\aa,\bb))_n}z^n+r_N(\aa,\bb;z),
\end{equation}
where
$$
r_N(\aa,\bb;z):=\frac{z^{\psi(\aa;\bb)-1}}{\Gamma(\psi(\aa,\bb))}\sum\limits_{n=N}^{\infty}\frac{g_n(\aa;\bb)}{(\psi(\aa,\bb))_n}z^n
$$
and $\psi(\aa,\bb)$ is defined by (\ref{eq:psi-defined}).  These coefficients are polynomials symmetric in the components of the vectors $\aa=(a_1,\ldots,a_{p-1})$ and $\bb=(b_1,\ldots,b_p)$.  They can also be defined via the inverse factorial generating function \cite[eq.(2.21)]{Norlund}
$$
\frac{\Gamma(z+\psi(\aa;\bb))\Gamma(z+\aa)}{\Gamma(z+\bb)}=\sum\limits_{n=0}^{\infty}\frac{g_n(\aa;\bb)}{(z+\psi(\aa,\bb))_n}.
$$
As, clearly, $\psi(\aa+\alpha;\bb+\alpha)=\psi(\aa;\bb)+\alpha$, we have (by changing $z\to{z+\alpha}$)
$$
\frac{\Gamma(z+\alpha+\psi(\aa;\bb))\Gamma(z+\alpha+\aa)}{\Gamma(z+\alpha+\bb)}=\sum\limits_{n=0}^{\infty}\frac{g_n(\aa+\alpha;\bb+\alpha)}{(z+\psi(\aa;\bb)+\alpha)_n}
=\sum\limits_{n=0}^{\infty}\frac{g_n(\aa;\bb)}{(z+\alpha+\psi(\aa;\bb))_n}.
$$
Hence, $g_n(\aa+\alpha;\bb+\alpha)=g_n(\aa;\bb)$ for any $\alpha$. N{\o}rlund found two different recurrence relations for $g_n$ (one in $p$ and one in $n$). The simplest of them reads \cite[eq.(2.7)]{Norlund}
\begin{equation}\label{eq:Norlundcoeff}
g_n(\aa,\alpha;\bb,\beta)=\sum\limits_{s=0}^{n}\frac{(\beta-\alpha)_{n-s}}{(n-s)!}(\psi(\aa;\bb)-\alpha+s)_{n-s}g_s(\aa;\bb),~~~p=1,2,\ldots,
\end{equation}
with the  initial values $g_0(-;b_1)=1$, $g_n(-;b_1)=0$, $n\ge1$.  This recurrence was solved by N{\o}rlund  \cite[eq.(2.11)]{Norlund} as follows:
\begin{subequations}
\begin{equation}\label{eq:Norlund-explicit}
g_n(\aa;\bb)=\sum\limits_{0\leq{j_{1}}\leq{j_{2}}\leq\cdots\leq{j_{p-2}}\leq{n}}
\prod\limits_{m=1}^{p-1}\frac{(\psi_m+j_{m-1})_{j_{m}-j_{m-1}}}{(j_{m}-j_{m-1})!}(b_{m+1}-a_{m})_{j_{m}-j_{m-1}},
\end{equation}
where $\psi_m=\sum_{i=1}^{m}(b_i-a_i)$, $j_0=0$, $j_{p-1}=n$.   This formula can be rewritten as:
\begin{equation}\label{eq:Norlund-explicit1}
g_n(\aa;\bb)=\sum\limits_{0\leq{j_{1}}\leq{j_{2}}\leq\cdots\leq{j_{p-2}}\leq{n}}
\prod\limits_{m=1}^{p-1}\frac{(-j_{m})_{j_{m-1}}(b_{m+1}-a_{m})_{j_{m}}(\psi_m)_{j_m}}{(1+a_{m}-b_{m+1})(\psi_m)_{j_{m-1}}j_{m}!},
\end{equation}
\end{subequations}
where we applied
$$
(\alpha)_{n-s}=\frac{(-1)^s(\alpha)_n}{(1-\alpha-n)_s},~~~~(n-s)!=\frac{(-1)^sn!}{(-n)_s},~~~(\alpha+s)_{n-s}=\frac{(\alpha)_n}{(\alpha)_s}.
$$
Note that the presence of the terms $(-j_{m})_{j_{m-1}}$ allows extending the above sums to $\mathbb{N}^{p-2}_{0}$ without changing their values. The other recurrence relation for $g_n(\aa;\bb)$ discovered by N{\o}rlund \cite[eq.(1.28)]{Norlund} has order $p$ in the variable $n$ and coefficients polynomial in $n$. Details can be found in \cite[section~2.2]{KPSIGMA}.  The first three coefficients  are given by (see \cite[Theorem~3.1]{KPSIGMA} for details):
$$
g_0(\aa;\bb)=1,~~~~g_1(\aa;\bb)=\sum_{m=1}^{p-1}(b_{m+1}-a_m)\psi_m,
$$
$$
g_2(\aa;\bb)=\frac{1}{2}\sum_{m=1}^{p-1}(b_{m+1}-a_m)_2(\psi_m)_2+\sum_{k=2}^{p-1}(b_{k+1}-a_{k})(\psi_{k}+1)\sum_{m=1}^{k-1}(b_{m+1}-a_m)\psi_m.
$$
For $p=2$ and $p=3$ and arbitrary $n$ explicit expressions for $g_n(\aa;\bb)$ have also been found by N{\o}rlund, see \cite[eq.(2.10)]{Norlund}. Defining $\nu_m:=\sum_{j=1}^{m}b_j-\sum_{j=1}^{m-1}a_j$, we have
\begin{equation}\label{eq:gnp2p3}
\begin{split}
&g_n(a;\bb)=\frac{(b_1-a)_{n}(b_2-a)_{n}}{n!}~~\text{for}~p=2;
\\
&g_n(\aa;\bb)=\frac{(\nu_3-b_2)_n(\nu_3-b_3)_n}{n!}
{}_3F_2\left(\!\!\begin{array}{l}-n,b_1-a_1,b_1-a_2\\\nu_3-b_2,\nu_3-b_3\end{array}\!\!\right)~~\text{for}~p=3.
\end{split}
\end{equation}
The right hand side here is invariant with respect to the permutation of the elements of $\bb$. Finally, for $p=4$ we have \cite[p.12]{KPSIGMA}
$$
g_n(\aa;\bb)=\frac{(\nu_4-b_3)_n(\nu_4-b_4)_n}{n!}
\sum\limits_{k=0}^{n}\frac{(-n)_k(\nu_2-a_{2})_k(\nu_2-a_{3})_k}{(\nu_4-b_3)_k(\nu_4-b_4)_k}
{}_{3}F_{2}\!\left(\!\begin{array}{l}-k,b_1-a_{1},b_2-a_{1}\\\nu_2-a_{2}, \nu_2-a_{3}\end{array}\!\!\right).
$$

The following lemma will play an important role in proving the convergence of the expansions considered in the sequel.
\begin{lemma}\label{lm:gnestimate}
Given two complex vectors  $\aa\in\C^{p-1}$ and $\bb\in\C^{p}$, denote by $-a$  the real part of the rightmost pole(s) of the function $s\to\Gamma(\aa+s)/\Gamma(\bb+s)$ and  write $r\in\N$ for the maximal multiplicity among all poles with the real part $-a$. Then for $n\ge2$ there exists a constant $K>0$ independent of $n$ such that
\begin{equation}\label{eq:Norlundestimate}
\left|\frac{g_n(\aa;\bb)}{\Gamma(\psi(\aa;\bb)+n)}\right|\leq\frac{K\log^{r-1}(n)}{n^{a+1}}.
\end{equation}
\end{lemma}
\begin{proof} Assume first that $a<0$. This implies that the rightmost pole(s) of the function $s\to\Gamma(\aa+s)/\Gamma(\bb+s)$ coincide with
the rightmost pole(s) of the function $s\to\Gamma(s)\Gamma(\aa+s)/\Gamma(\bb+s)$. Define
$$
f(w):=w^{1-\psi(\aa;\bb)}G^{p,0}_{p,p}\!\left(1-w~\vline\begin{array}{l}\bb\\\aa,0\end{array}\right)
=\sum\limits_{n=0}^{\infty}\frac{g_n(\aa;\bb)}{\Gamma(\psi(\aa;\bb)+n)}w^n
$$
by (\ref{eq:Norlund}). It follows from the properties of the Meijer-N{\o}rlund function $G^{p,0}_{p,p}$ that $f(w)$ is  analytic in the domain
$$
\Delta(\phi,\eta):=\{w:|w|<1+\eta,|\arg(1-w)|\ge\phi\}
$$
for some $\eta>0$ and $0<\phi<\pi/2$.  Further, from the asymptotic properties of $G^{p,0}_{p,p}(z)$ in the neighborhood of $z=0$ given in \cite[Property~5]{KLJAT2017} we conclude that
$$
f(w)=\mathcal{O}\left((1-w)^{a}\log^{r-1}(1-w)\right)~~~\text{as}~w\to1,
$$
where $a$ and $r$ are as defined in the Lemma. Hence, we are in position to apply \cite[Theorem~2]{FO} stating that
$$
\left|\frac{g_n(\aa;\bb)}{\Gamma(\psi(\aa;\bb)+n)}\right|={\mathcal O}\left(n^{-a-1}\log^{r-1}(n)\right)~\text{as}~n\to\infty.
$$
Next, assume that $a\ge0$.  Take $\beta>0$ large enough to make the real part of the rightmost pole(s) of the function $s\to\Gamma(\aa-\beta+s)/\Gamma(\bb-\beta+s)$ positive and denote this real part $-a'$, so that $a'<0$ and we are in the situation treated above.  Hence, using $a'=a-\beta$, $g_n(\aa-\beta;\bb-\beta)=g_n(\aa;\bb)$ and $\psi(\aa-\beta;\bb-\beta)=\psi(\aa;\bb)-\beta$ we get
\begin{multline*}
\left|\frac{g_n(\aa;\bb)}{\Gamma(\psi(\aa;\bb)+n)}\right|=\left|\frac{g_n(\aa-\beta;\bb-\beta)\Gamma(\psi(\aa-\beta;\bb-\beta)+n)}{\Gamma(\psi(\aa-\beta;\bb-\beta)+n)\Gamma(\psi(\aa;\bb)+n)}\right|
\\
\leq\frac{K_1\log^{r-1}(n)}{n^{a'+1}}\left|\frac{\Gamma(\psi(\aa;\bb)-\beta+n)}{\Gamma(\psi(\aa;\bb)+n)}\right|
\leq\frac{K_1\log^{r-1}(n)}{n^{a-\beta+1}}K_2n^{-\beta}=\frac{K_1K_2\log^{r-1}(n)}{n^{a+1}}.
\end{multline*}
\end{proof}

\noindent\textbf{Remark.} In some situations below we will need an extended version of inequality (\ref{eq:Norlundestimate}) valid for all $n\ge0$.
It is straightforward to see that (\ref{eq:Norlundestimate}) implies that the inequality
\begin{equation}\label{eq:Norlundestimate0}
\left|\frac{g_n(\aa;\bb)}{\Gamma(\psi(\aa;\bb)+n+\mu)}\right|\leq C\frac{\log^{r-1}(n+2)}{(n+1)^{a+1+\mu}}
\end{equation}
is true with some positive constant $C$ for all $n\ge0$ and any $\mu\in\R$.

\begin{lemma}\label{lm:remainder}
Suppose $a$ and $r$ retain their meaning from Lemma~\ref{lm:gnestimate} and assume further that $a>0$.  Then for $x\in[0,1]$ the remainder $r_N(\aa,\bb;x)$ in formula \eqref{eq:Norlund} satisfies
\begin{equation}\label{eq:rNestimate}
\left| r_N(\aa,\bb;x)\right|\leq\frac{K\log^{r-1}(N)}{N^{a}}x^{N+\psi(\aa;\bb)-1}
\end{equation}
for some constant $K>0$ independent of $N$ and $x$.
\end{lemma}
\begin{proof}
For $x\in[0,1]$, and using the previous lemma we have
$$
\left| r_N(\aa,\bb;x)\right|\leq Kx^{N+\psi(\aa;\bb)-1}s_N(a,r), \hskip 5mm s_N(a,r):=\sum_{n=N}^\infty{\log^{r-1}n\over n^{a+1}}=(-1)^{r-1}{d^{r-1}\over da^{r-1}}\sum_{n=N}^\infty{1\over n^{a+1}}.
$$
Then
\begin{multline*}
s_N(a,r)=(-1)^{r-1}{d^{r-1}\over da^{r-1}}\zeta(a+1,N)
\\
={d^{r-1}\over da^{r-1}}{(-1)^{r-1}\over\Gamma(a+1)}\int_0^\infty{t^ae^{-Nt}\over 1-e^{-t}}dt\le\sum_{k=0}^{r-1} c_k(a,r)\int_0^\infty{t^a\vert \log^k t\vert e^{-Nt}\over 1-e^{-t}}dt,
\end{multline*}
where $\zeta(a+1,N)$ is the Hurwitz zeta function \cite[section~25.11]{NIST}, whose integral representation  \cite[eq.(25.11.25)]{NIST} was used in the second equality.  Here the constants $c_k(a,r)>0$ are independent of $N$. Further,
\begin{multline*}
\int_0^\infty{t^a\vert\log^k t\vert e^{-Nt}\over 1-e^{-t}}dt\le \int_0^\infty t^{a-1}(1+t)\vert\log^k t\vert e^{-Nt}dt
\\
={1\over N^{a}}\int_0^\infty t^{a-1}\left(1+\frac{t}{N}\right)\vert\log^k\left(\frac{t}{N}\right)\vert e^{-t}dt.
\end{multline*}
The term $\vert\log^k(t/N)\vert$ is bounded by a sum of terms $\vert\log^p t\log^q N\vert$, with $p+q=k$, the highest one corresponding to $q=k=r-1$:
$$
{\log^{r-1} N\over N^{a}}\int_0^\infty t^{a-1}\left(1+\frac{t}{N}\right) e^{-t}dt={\mathcal O}\left({\log^{r-1} N\over N^{a}}\right).
$$
\end{proof}

\section{Expansions for the Bessel type GHF}

\subsection{An expansion in terms of Bessel functions}

\begin{theorem}\label{th:GHFBesselBessel}
For $\aa\in\C^{p-1}$ with $\Re\aa>0$, $\bb\in\C^{p}$ and $z\in\C$, let $a$ and $r$ be the constants defined in Lemma~\ref{lm:gnestimate}.
Then, if $\Re\psi(\aa;\bb)>1/2$, for any $N\in\N$ we have
\begin{equation}\label{expJtwo}
{}_{p-1}F_p\left.\left(\begin{matrix}\aa\\\bb\end{matrix}\:\right\vert -\frac{z^2}{4}\right)
=\frac{\Gamma(\bb)}{\Gamma(\aa)}\sum\limits_{n=0}^{N-1}g_n(\aa;\bb)\frac{J_{\psi(\aa;\bb)+n-1}(z)}{(z/2)^{\psi(\aa;\bb)+n-1}}+R_N(z),
\end{equation}
where
\begin{equation}\label{bound1}
\vert R_N(z)\vert\le Ke^{\vert\Im z\vert}\frac{\log^{r-1}(N)}{N^{a+1/2}}
\end{equation}
with $K>0$ independent of $N$ and $z$. Therefore, \eqref{expJtwo} converges uniformly with respect to $z$ in any horizontal strip $\vert\Im z\vert<\Lambda$ with arbitrary $\Lambda>0$.
\end{theorem}
\begin{proof}
Substituting the expansion (\ref{eq:Norlund}) into formula (\ref{ini}) and integrating term-wise we get
\begin{equation}\label{expJone}
\begin{split}
{}_{p-1}F_p\left.\left(\begin{matrix}\aa\\\bb\end{matrix}\:\right\vert -\frac{z^2}{4}\right)
&=\frac{\Gamma(\bb)}{\Gamma(\aa)}\sum\limits_{n=0}^{N-1}\frac{g_n(\aa;\bb)}{\Gamma(\psi(\aa;\bb)+n)}{}_0F_{1}\left(-;\psi(\aa;\bb)+n;-\frac{z^2}{4}\right)
+R_N(z),
\\
R_N(z)&:= \frac{2\Gamma(\bb)}{\sqrt{\pi}\Gamma(\aa)}\int_0^1\cos(zu)r_N(\aa,\bb;1-u^2)\,du,
\end{split}
\end{equation}
where we used the shifting property $g_n(\aa-1/2;\bb-1/2)=g_n(\aa;\bb)$ and the integral evaluation
$$
\int_0^1(1-u^2)^{\psi(\aa;\bb)+n-3/2}\cos(zu)du=\frac{\sqrt{\pi}\Gamma(\psi(\aa;\bb)+n-1/2)}{2\Gamma(\psi(\aa;\bb)+n)}{}_0F_{1}\left(-;\psi(\aa;\bb)+n;-\frac{z^2}{4}\right),
$$
valid for $n\in\N_0$ when $\Re\psi(\aa;\bb)>1/2$.  Expansion \eqref{expJone} can be rewritten in the form \eqref{expJtwo}, in terms of Bessel functions, in view of
$$
J_{\nu}(z)=\frac{(z/2)^{\nu}}{\Gamma(\nu+1)}{}_0F_{1}\left(-;\nu+1;-\frac{z^2}{4}\right).
$$
From Lemma~\ref{lm:remainder} we have that
$$
\vert R_N(z)\vert\le C\,e^{\vert\Im z\vert}\frac{\log^{r-1}(N)}{N^a}\int_0^1(1-u^2)^{N+\psi(\aa;\bb)-1}du\le Ke^{\vert\Im z\vert}\frac{\log^{r-1}(N)}{N^{a+1/2}},
$$
with $C,K>0$ independent of $z$ and $N$, which is \eqref{bound1}.
\end{proof}

The following picture illustrates the accuracy and uniform character of approximation \eqref{expJtwo} for $p=2$, ${\bf a}=3$, ${\bf b}=(7/2,5)$ and real $z$.

\begin{figure}[H]
  \centering
   \includegraphics[width=6cm]{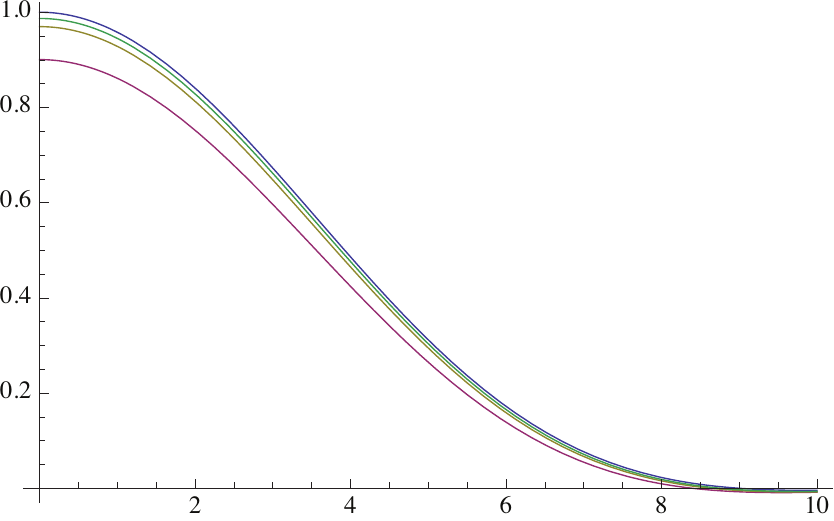}
\hskip 2cm
  \includegraphics[width=6cm]{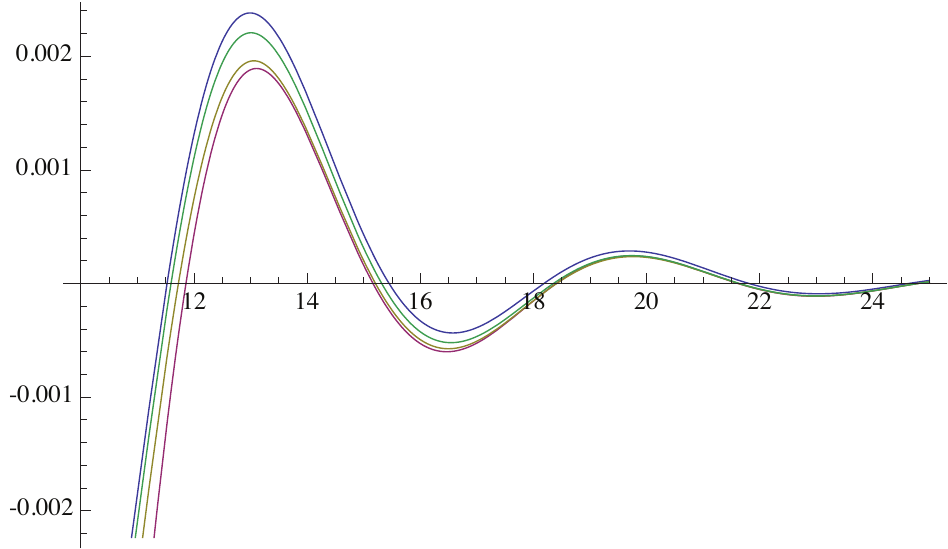}
  \caption{\small Plot of the left (blue) and right hand side of \eqref{expJtwo} on two intervals of the real axis. For $N=1,3,5$ (red, yellow and green, respectively) in the left picture and $N=1,10,20$ (red, yellow and green, respectively) in the right picture.}\label{figure2}
\end{figure}
\noindent

\textbf{Remark}. In view of the first formula in (\ref{eq:gnp2p3}), expansion (\ref{expJtwo}) for $p=2$ takes the form
$$
{}_{1}F_2\left.\left(\begin{matrix}a\\b_1,b_2\end{matrix}\:\right\vert -z^2/4\right)
=\frac{\Gamma(b_1)\Gamma(b_2)}{(z/2)^{\psi-1}\Gamma(a)}\sum_{n=0}^{\infty}\frac{(b_1-a)_n(b_2-a)_n}{n!(z/2)^{n}}J_{n+\psi-1}(z),
$$
where $\psi=b_1+b_2-a$.  Surprisingly, we could not find the above expansion in \cite{PBM2}.
On the other hand, \cite[eq.(5.7.8.3)]{PBM2} reads (after some change of notation):
$$
{}_1F_2\left.\left(\begin{matrix}a\\b_1,b_2\end{matrix}\:\right\vert -z^2/4\right)
=\frac{\Gamma(b_2)}{(z/2)^{b_2-1}}\sum\limits_{n=0}^{\infty}\frac{(b_1-a)_n}{(b_1)_nn!}(z/2)^nJ_{n+b_2-1}(z).
$$
As we have $(b_1)_nn!$ in the denominator, for small $z$ this series converges  faster than our expansion.
However, this expansion is not uniform in $z$ in any unbounded domain.

\subsection{An expansion in terms of elementary functions}

\begin{theorem}\label{th:GHFBesselelementary}
For $\aa\in\C^{p-1}$ with $\Re\aa>0$, $\bb\in\C^{p}$ and $z\in\C$, let $a$ and $r$ be the constants defined in Lemma~\ref{lm:gnestimate}. Then, if $\Re\psi(\aa;\bb)>1/2$,
for any $N\in\N$ we have
\begin{multline}\label{exptwos}
{_{p-1}F_{p}}\left(\left.\!\!\begin{array}{c}
\aa\\\bb\end{array}\right|-\frac{z^2}{4}\!\right)
=\frac{2\Gamma(\bb)}{\sqrt{\pi}\Gamma(\aa)}\sum\limits_{n=0}^{N-1}\frac{g_n(\aa;\bb)}{\Gamma(\psi(\aa;\bb)+n-1/2)}
\biggl\{P_{m}(z,\psi(\aa;\bb)+n-1)\frac{\sin{z}}{z}\\ -Q_{m}(z,\psi(\aa;\bb)+n-1)\cos{z}\biggr\}+R^T_N(z), \hskip 2cm m:=N+\lfloor\Re\psi(\aa;\bb)-3/2\rfloor,
\end{multline}
where $P_m(z,\nu)$ and $Q_m(z,\nu)$ are the following rational functions of $z$:
\begin{equation}\label{PQ}
\begin{array}{cl}
P_m(z,\nu):=&\displaystyle{\sum_{j=0}^m\frac{a_{m,j}(\nu)}{(-z^2)^j}, \hskip 2cm a_{m,j}(\nu):=\sum_{k=j}^m
\frac{(1/2-\nu)_k(2k)!}{k!(2(k-j))!},}\\\\\
Q_m(z,\nu):=&\displaystyle{\sum_{j=1}^m\frac{b_{m,j}(\nu)}{(-z^2)^j}, \hskip 2cm b_{m,j}(\nu):=\sum_{k=j}^m
\frac{(1/2-\nu)_k(2k)!}{k!(2(k-j)+1)!}}.
\end{array}
\end{equation}
The remainder is bounded in the form
\begin{equation}\label{cotaT}
\vert R^T_N(z)\vert\le Ke^{\vert\Im z\vert}\left[\frac{1}{N^{\Re\psi(\aa;\bb)-1/2}}+\frac{\log^{r-1}(N)}{N^{a+1/2}}\right],
\end{equation}
with $K>0$ independent of $N$ and $z$. Therefore, for $\Re\psi(\aa;\bb)>1/2$, \eqref{exptwos} converges uniformly with respect to $z$ in any horizontal
strip $\vert\Im z\vert<\Lambda$ with arbitrary $\Lambda>0$.
\end{theorem}

\begin{proof}
From \cite[eq.(9)]{Lopez} we have that, for $m=1,2,\ldots$, the Bessel function ${}_0F_{1}$ may be written as
\begin{equation}\label{expanJz}
\frac{\sqrt{\pi}}{2}\frac{\Gamma(\nu+1/2)}{\Gamma(\nu+1)}{}_0F_{1}\left(-;\nu+1;-\frac{z^2}{4}\right)=P_{m-1}(z,\nu)\frac{\sin z}{z}-Q_{m-1}(z,\nu)\cos z +r_m(z,\nu)
\end{equation}
with $P_{m-1}(z,\nu)$ and $Q_{m-1}(z,\nu)$ given in \eqref{PQ} and, if $m>\Re\nu-1/2$, the remainder $r_m(z,\nu)$ is bounded as follows:
\begin{equation}\label{cotaJ}
\vert r_m(z,\nu)\vert\le\frac{Ke^{\vert\Im z\vert}}{m^{\Re\nu+1/2}},
\end{equation}
with $K>0$ independent of $z$ and $m$.  Therefore, the remainder $r_m(z,\nu)$ is asymptotically equivalent to  $m^{-\Re\nu-1/2}$ as $m\to\infty$ uniformly in $z$ in any fixed horizontal strip.

Hence, substituting \eqref{expanJz} into \eqref{expJone} we obtain \eqref{exptwos} with
$$
R^T_N(z):=\frac{2\Gamma(\bb)}{\sqrt{\pi}\Gamma(\aa)}\sum\limits_{n=0}^{N-1}\frac{g_n(\aa;\bb)}{\Gamma(\psi(\aa;\bb)+n-1/2)}r_m(z,\psi(\aa;\bb)+n-1)+R_N(z)
$$
with $R_N(z)$ as in Theorem~\ref{th:GHFBesselBessel}. Here we need to choose $m=m(N,n)$ to make
$$
\sum\limits_{n=0}^{N-1}\frac{g_n(\aa;\bb)}{\Gamma(\psi(\aa;\bb)+n-1/2)}r_m(z,\psi(\aa;\bb)+n-1)
$$
converge to zero as $N\to\infty$. Using the estimates (\ref{eq:Norlundestimate0}) and \eqref{cotaJ}, with $m>n+\Re\psi(\aa;\bb)-3/2$ we get
$$
\left|\sum\limits_{n=0}^{N-1}\frac{g_n(\aa;\bb)}{\Gamma(\psi(\aa;\bb)+n-1/2)}
r_m(z,\psi(\aa;\bb)+n-1)\right|\leq
Ce^{\vert\Im z\vert}\sum\limits_{n=0}^{N-1}\frac{\log^{r-1}(n+2)}{(n+1)^{a+1/2}m^{n+\Re\psi(\aa;\bb)-1/2}}.
$$

Then, it is sufficient to take $m=N+\lfloor\Re\psi(\aa;\bb)-3/2\rfloor$ to obtain
\begin{multline*}
\left|\sum\limits_{n=0}^{N-1}\frac{g_n(\aa;\bb)}{\Gamma(\psi(\aa;\bb)+n-1/2)}
r_m(z,\psi(\aa;\bb)+n-1)\right|\leq
\\
\frac{Ce^{\vert\Im z\vert}}{N^{\Re\psi(\aa;\bb)-1/2}}\sum\limits_{n=0}^{N-1}\frac{\log^{r-1}(n+2)}{(n+1)^{a+1/2}N^n}\le\frac{Ke^{\vert\Im z\vert}}{N^{\Re\psi(\aa;\bb)-1/2}},
\end{multline*}
with $K>0$ independent of $z$ and $N$ and \eqref{cotaT} follows.
\end{proof}

Formula \eqref{ejemplo} is a particular case of \eqref{exptwos} for $N=2$. The following picture shows some more approximations.

\begin{figure}[H]
  \centering
   \includegraphics[width=6cm]{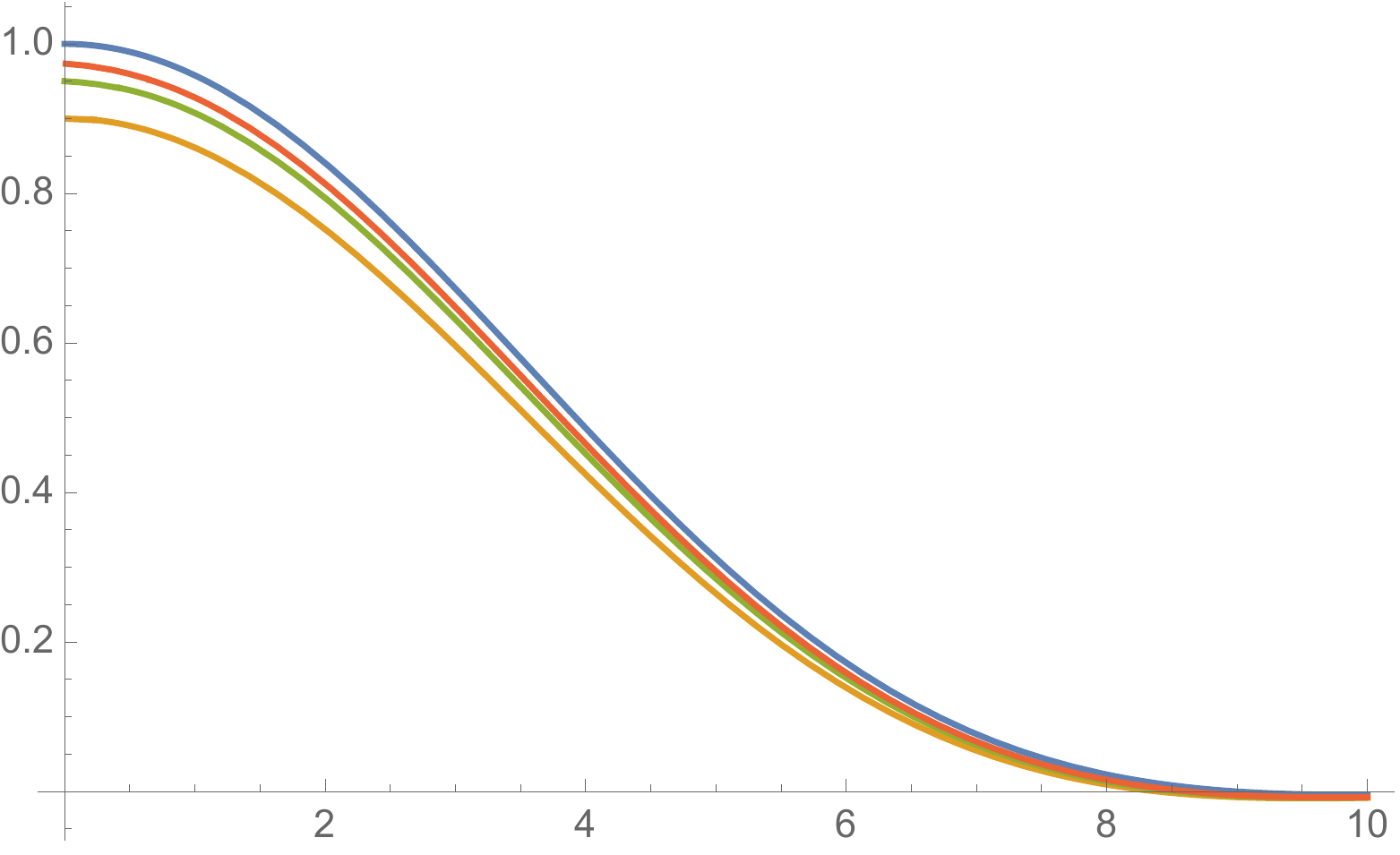}
\hskip 2cm
  \includegraphics[width=6cm]{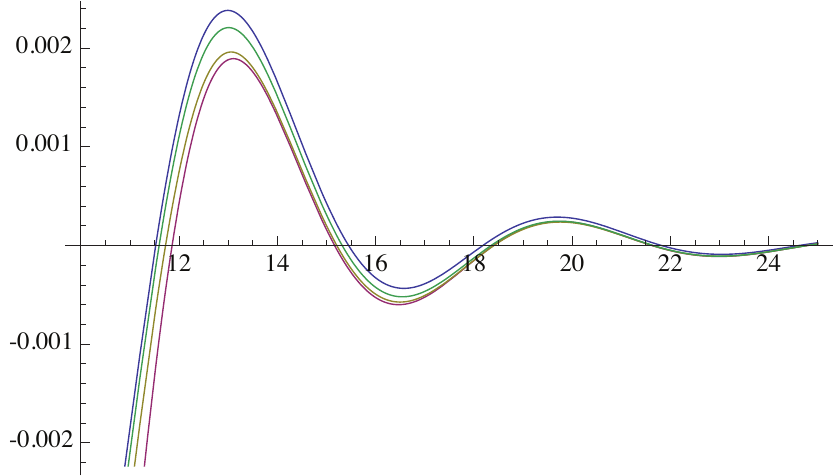}
  \caption{\small Plot of the left (blue) and right hand side of \eqref{exptwos} on two intervals of the real axis. For $N=1,3,5$ (red, yellow and green, respectively) in the left picture and $N=1,10,20$ (red, yellow and green, respectively) in the right picture.}\label{figure3}
\end{figure}
\noindent

\section{Expansions for the GHF of the Kummer type}

\subsection{An expansion in terms of the Kummer functions}

\begin{theorem}\label{th:pFpKummer}
Given $\aa\in\C^{p}$, assume without loss of generality that $\Re(a_p)=\min(\Re(\aa))$ and suppose that
$\Re(\aa_{[p]})>0$. Suppose further that $\bb\in\C^{p}$, $z\in\C$ and denote by $-\alpha$ the real part of the rightmost pole(s) of the function $s\to\Gamma(\aa_{[p]}+s)/\Gamma(\bb+s)$ and by $r\in\N$ the maximal multiplicity among all poles with the real part $-\alpha$.  Then, if $\Re(\psi(\aa;\bb))>0$, for any $N\in\N$ we have
\begin{equation}\label{expKtwo}
\left.{}_{p}F_p\left(\begin{matrix}\aa\\ \bb\end{matrix}\right\vert -z\right)=
{\Gamma(\bb)\over\Gamma(\aa_{[p]})}\,\sum_{n=0}^{N-1} \frac{g_n(\aa_{[p]};\bb)}{\Gamma(\psi(\aa_{[p]};\bb)+n)}\,M(a_p,\psi(\aa_{[p]};\bb)+n,-z)+R_N(z),
\end{equation}
where $M(a,b,z)$ is the Kummer function of the first kind, and the remainder is bounded in the form
\begin{equation}\label{cotaK}
\vert R_N(z)\vert\le KH(z)\frac{\log^{r-1}N}{N^{\alpha}},
\end{equation}
with $K>0$ independent of $N$ and $z$ and $H(z):=\max(1,e^{-\Re z})$. Therefore, expansion \eqref{expKtwo} is uniformly convergent for $z$ in any half-plane $\Re{z}\ge\Lambda$ with arbitrary $\Lambda\in\R$.
\end{theorem}

\begin{proof}
Assume without loss of generality that $\Re(a_p)=\min(\Re(\aa))$ (otherwise just exchange the indices $p$ and $\mathrm{argmin}(\aa))$.  The integral representation \cite[eq.(11)]{dimitri} of ${}_{p}F_p$ combined with \cite[Sec. 16.19, eq. (16.19.2)]{hyper} and the shifting property (\ref{eq:Gtimespower}) gives
\begin{equation}\label{ini2}
\left.{}_{p}F_p\left(\begin{matrix}\aa\\ \bb\end{matrix}\right\vert -z\right)=
{\Gamma(\bb)\over\,\Gamma(\aa)}\int_0^1e^{-zt}
G_{p,p}^{p,0}\left(t\left| \begin{matrix}\bb-1\\ \aa-1\end{matrix}\right.\right)dt
={\Gamma(\bb)\over\,\Gamma(\aa)}\int_0^1e^{-zt}t^{a_p-1}
G_{p,p}^{p,0}\left(t\left| \begin{matrix}\bb'\\ \aa',0\end{matrix}\right.\right)dt,
\end{equation}
where $\bb'=\bb-a_p$, $\aa'=\aa_{[p]}-a_p$.  Then, we are in the position to apply expansion (\ref{eq:Norlund})
to get:
$$
\left.{}_{p}F_p\left(\begin{matrix}\aa\\ \bb\end{matrix}\right\vert -z\right)
={\Gamma(\bb)\over\,\Gamma(\aa)}\sum\limits_{n=0}^{N-1}
\frac{\Gamma(a_p)\,\,g_n(\aa';\bb')}{\Gamma(\psi(\aa';\bb')+n)}\,M(a_p,\psi(\aa';\bb')+n+a_p,-z)+R_N(z),
$$
where
$$
R_N(z)={\Gamma(\bb)\over\,\Gamma(\aa)}\int_0^1e^{-zt}t^{a_p-1}(1-t)^{\psi(\aa';\bb')-1}dt\sum\limits_{n=N}^{\infty}
\frac{g_n(\aa';\bb')}{\Gamma(\psi(\aa';\bb')+n)}(1-t)^n.
$$
Now by the shifting property of N{\o}rlund's coefficients and definition (\ref{eq:psi-defined}) we have
$$
g_n(\aa';\bb')=g_n(\aa_{[p]};\bb)~~\text{and}~~\psi(\aa';\bb')=\psi(\aa;\bb),
$$
so that we arrive at expansion (\ref{expKtwo}) with the remainder given by (after termwise integration)
$$
R_N(z)={\Gamma(\bb)\over\,\Gamma(\aa)}\sum\limits_{n=N}^{\infty}
\frac{g_n(\aa_{[p]};\bb)}{\Gamma(\psi(\aa;\bb)+n)}\int_0^1e^{-zt}t^{a_p-1}(1-t)^{\psi(\aa;\bb)+n-1}dt.
$$
In view of  $\psi(\aa;\bb)+a_p=\psi(\aa_{[p]};\bb)$, we have
\begin{multline*}
\left|\int_0^1e^{-zt}t^{a_p-1}(1-t)^{\psi(\aa;\bb)+n-1}dt\right|\leq\int_0^1e^{-\Re(z)t}t^{\Re(a_p)-1}(1-t)^{\Re(\psi(\aa;\bb))+n-1}dt \\
\le H(z)\frac{\Gamma(\Re(a_p))\Gamma(\Re(\psi(\aa;\bb))+n)}{\Gamma(\Re(\psi(\aa_{[p]};\bb))+n)},
\end{multline*}
where $H(z)$ is defined in the statement of the theorem. Now, if we write $-\alpha$ for the real part of the rightmost pole(s) of the function $s\to\Gamma(\aa_{[p]}+s)/\Gamma(\bb+s)$ and   $r\in\N$ for the maximal multiplicity among all poles with the real part $-\alpha$, Lemma~\ref{lm:gnestimate} yields
$$
\left\vert\frac{g_n(\aa_{[p]};\bb)}{\Gamma(\psi(\aa;\bb)+n)}\right\vert
=\left\vert\frac{g_n(\aa_{[p]};\bb)\Gamma(\psi(\aa_{[p]};\bb)+n)}{\Gamma(\psi(\aa_{[p]};\bb)+n)\Gamma(\psi(\aa;\bb)+n)}\right\vert
\le\frac{K_1\log^{r-1}(n)}{n^{\alpha+1}}\left\vert\frac{\Gamma(\psi(\aa;\bb)+a_p+n)}{\Gamma(\psi(\aa;\bb)+n)}\right\vert.
$$
Combining these bounds we obtain
$$
|R_N(z)|\le {K}H(z)|\Gamma(\Re(a_p))|\left|{\Gamma(\bb)\over\,\Gamma(\aa)}\right|\sum\limits_{n=N}^{\infty}\lambda_n\frac{\log^{r-1}(n)}{n^{\alpha+1}}
$$
for some positive constant $K$ independent of $n$ and $z$, and
$$
\lambda_n:=\left\vert\frac{\Gamma(\psi(\aa_{[p]};\bb)+n)}{\Gamma(\psi(\aa;\bb)+n)}\frac{\Gamma(\Re(\psi(\aa;\bb))+n)}{\Gamma(\Re(\psi(\aa_{[p]};\bb))+n)}\right\vert
$$
The representation \cite[eq.1.3(3)]{Bateman} ($\gamma$ is the Euler-Mascheroni constant)
$$
\frac{\Gamma(x+iy)}{\Gamma(x)}=\frac{xe^{i\gamma{y}}}{x+iy}\prod_{k=1}^{\infty}\frac{e^{i\gamma/k}}{1+i\gamma/(k+x)},
$$
valid for all $x+iy\in\C$, implies that the sequence $\lambda_n$ is bounded by a constant.  Hence by Lemma~\ref{lm:remainder} we get the bound (\ref{cotaK}).
\end{proof}

The following picture illustrates the accuracy and the uniform character of approximation \eqref{expKtwo} for $p=2$, ${\bf a}=(1,3/2)$, ${\bf b}=(2,3)$ and real $z$.

\begin{figure}[H]
  \centering
   \includegraphics[width=6cm]{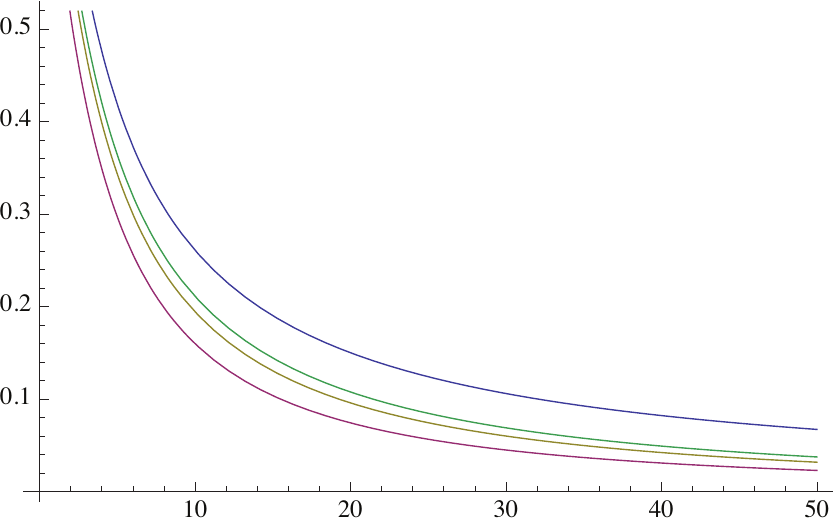}
\hskip 2cm
  \includegraphics[width=6cm]{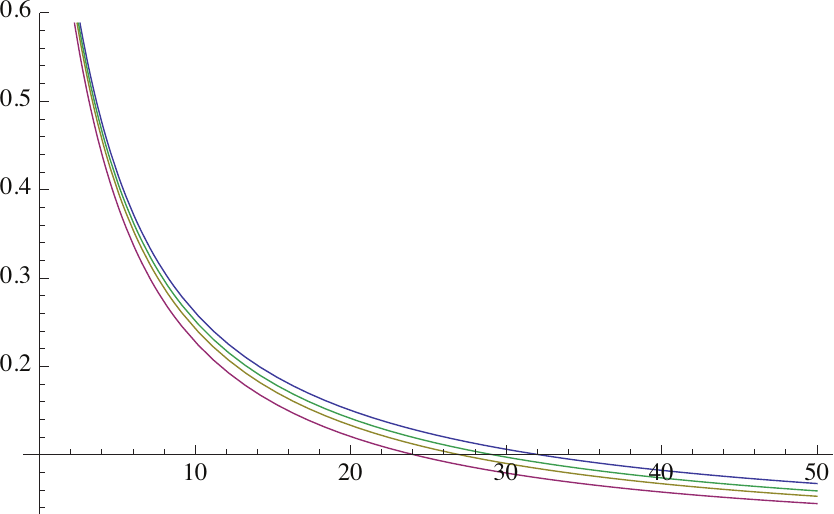}
  \caption{\small Plot of the left (blue) and right hand side of \eqref{expKtwo} for $z$ in the real interval $[0,50]$. For $N=10,20,30$ (red, yellow and green, respectively) in the left picture and $N=50,100,200$ (red, yellow and green, respectively) in the right picture.}\label{figure4}
\end{figure}
\noindent

\subsection{An expansion in terms of elementary functions}

\begin{theorem}\label{th:GHFKummerelementary}
Given $\aa,\bb\in\C^{p}$, assume that $\Re(\aa)>0$ and let $\alpha$ and $r$ have the same meaning as in Theorem~\ref{th:pFpKummer}.
Then, for $\Re\psi(\aa;\bb)>0$ for any $N\in\N$ we have
\begin{multline}\label{exptwo}
{_{p}F_{p}}\left(\left.\!\!\begin{array}{c}
\aa\\\bb\end{array}\right|-z\!\right)
=\frac{\Gamma(\bb)}{\Gamma(\aa_{[p]})}\sum\limits_{n=0}^{N}\frac{g_n(\aa_{[p]};\bb)}{\Gamma(\psi(\aa_{[p]};\bb)+n)}
\left[\sum_{k=0}^{m-1}A_k(a_p,\psi(\aa_{[p]};\bb)+n)F_k(-z)\right]+R^T_N(z),
\end{multline}
with $m:=N+\lfloor\Re\psi(\aa_{[p]};\bb)\rfloor$, and
$$
F_n(z):=\frac{n!}{(-z)^{n+1}}\left[e_n\left(\frac{z}{2}\right)-e^{z}e_n\left(-\frac{z}{2}\right)\right],\hspace{2cm} e_n(z):= \sum_{k=0}^{n}\frac{z^k}{k!},
$$
$$
A_n(a,b):=2^{n+2-b}\frac{(a+1-b)_n}{n!}\frac{\Gamma(b)}{\Gamma(a)\Gamma(b-a)}{}_2F_1\left(\left.
\begin{array}{cl}
1-a, \hskip 2mm -n\\\\
b-a-n
\end{array}
\right\vert -1\right).
$$
The remainder term is bounded in the form
\begin{equation}\label{cotaKK}
\left|R^T_N(z)\right|\leq K\,H(z)\left[\frac{\log^{r-1}N}{N^{\alpha}}+\frac{1}{N^{\Re a_p}}\right],
\end{equation}
with $K>0$ independent of $z$ and $N$, $H(z):=\max(1,e^{-\Re z})$. Therefore, expansion \eqref{exptwo} is uniformly convergent any half plane $\Re z\ge\Lambda$ with arbitrary  $\Lambda\in\mathbb{R}$.
\end{theorem}

\begin{proof}
According to \cite[(21),(28)]{confluent} for each $n=1,2\ldots$, we can write the Kummer function $M(a,b;z)$ in the form:
\begin{equation}\label{expanM}
M(a,b;-z)=\sum_{k=0}^{m-1}A_k(a,b)F_k(-z)+r_m(a,b;-z),
\end{equation}
with $A_k(a,b)$ and $F_k(z)$ defined above. We also have that, for $m>\Re(b-1)$, the remainder $r_m(a,b;z)$ is bounded in the form:
\begin{equation}\label{cotaM}
\vert r_m(a,b;-z)\vert\le\frac{KH(z)}{m^{\beta}}, \hskip 2cm \beta:=\min\lbrace\Re a,\Re(b-a)\rbrace,
\end{equation}
with $K>0$ independent of $z$ and $n$ and $H(z)=\max(1,e^{-\Re{z}})$.  Therefore, the remainder $r_m(a,b;z)$ behaves as $m^{-\beta}$ as $m\to\infty$ uniformly in $z$ in any half-plane of the form  $\Re{z}\ge\Lambda$, $\Lambda\in\R$.

Hence, substituting \eqref{expanM} into \eqref{expKtwo} we obtain \eqref{exptwo} with
$$
R^T_N(z):=\frac{\Gamma(\bb)}{\Gamma(\aa_{[p]})}\sum\limits_{n=0}^{N-1}\frac{g_n(\aa_{[p]};\bb)}{\Gamma(\psi(\aa_{[p]};\bb)+n)}r_m(a_p,\psi(\aa_{[p]};\bb)+n;-z)+R_N(z).
$$
Next, we need to choose $m=m(N,n)$ to make
$$
\sum\limits_{n=0}^{N-1}\frac{g_n(\aa_{[p]};\bb)}{\Gamma(\psi(\aa_{[p]};\bb)+n)}r_m(a_p,\psi(\aa_{[p]};\bb)+n;-z)
$$
converge to zero as $N\to\infty$.

Assuming that $m>\Re(\psi(\aa_{[p]};\bb))+n-1$, we can use the estimate \eqref{cotaM} with $\beta=\Re(a_p)$ in view of $\Re(\psi(\aa_{[p]};\bb))>0$ and estimate (\ref{eq:Norlundestimate0}) to get
$$
\left|\sum\limits_{n=0}^{N-1}\frac{g_n(\aa_{[p]};\bb)}{\Gamma(\psi(\aa_{[p]};\bb)+n)}r_m(a_p,\psi(\aa_{[p]};\bb)+n;-z)\right|\leq
CH(z)\sum\limits_{n=0}^{N-1}\frac{\log^{r-1}(n+2)}{(n+1)^{\alpha+1}m^{\Re a_p}}.
$$
Then, it is sufficient to take $m=N+\lfloor\Re\psi(\aa_{[p]};\bb)\rfloor$ and we obtain
$$
\left|\sum\limits_{n=0}^{N-1}\frac{g_n(\aa_{[p]};\bb)}{\Gamma(\psi(\aa_{[p]};\bb)+n)}r_m(a_p,\psi(\aa_{[p]};\bb)+n;-z)\right|\leq
C\frac{H(z)}{N^{\Re a_p}}\sum\limits_{n=0}^{N-1}\frac{\log^{r-1}(n+2)}{(n+1)^{\alpha+1}}\le K\frac{H(z)}{N^{\Re a_p}},
$$
with $C,K>0$ independent of $z$ and $N$ which implies \eqref{cotaKK}.
\end{proof}

The following picture illustrates the accuracy and the uniform character of approximation \eqref{exptwo} for $p=2$, ${\bf a}=(1,3/2)$, ${\bf b}=(2,3)$ and real $z$.

\begin{figure}[H]
  \centering
   \includegraphics[width=6cm]{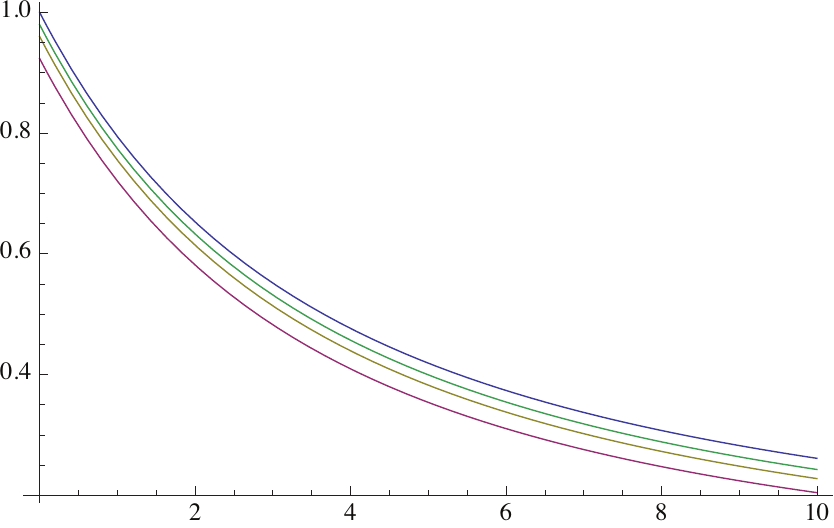}
\hskip 2cm
  \includegraphics[width=6cm]{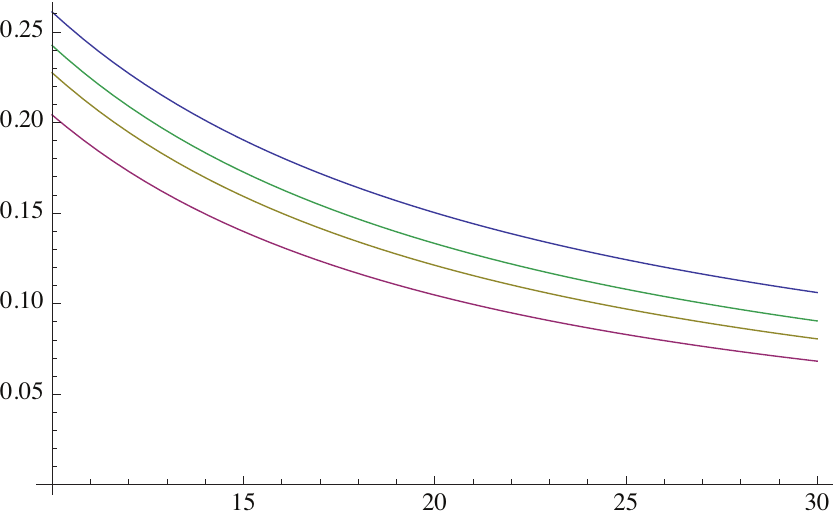}
  \caption{\small Plot of the left (blue) and right hand side of \eqref{exptwo} for $N=20,40,80$ (red, yellow and green, respectively) for $z$ in two different real intervals.}\label{figure5}
\end{figure}
\noindent

\bigskip
\bigskip

\section{Acknowledgments}

The authors L\'opez and Pagola acknowledge the {\it Direcci\'on General de Ciencia y Tecnolog\'{\i}a} (REF. MTM2017-83490-P) for its financial support.

\footnotesize{

}


\begin{thebibliography}{5}

\bibitem{AAR} G.E.\:Andrews, R.\:Askey and R.\:Roy, \textit{Special functions}, Cambridge University Press, 1999.

\bibitem{hyper} {R.A.\:Askey and A.B.\:Olde Daalhuis}, Generalized Hypergeometric Functions and Meijer G-Function, in: \textit{NIST Handbook of Mathematical Functions}, Cambridge University Press, Cambridge, 2010, pp. 403--418 (Chapter 16).

\bibitem{Bateman} H.\:Bateman, A.\:Erd\'{e}lyi, Higher transcedental functions, Volume I, MacGrow Hill Book Company, 1953.

\bibitem{BealsWong} R.\:Beals and R.\:Wong, \textit{Special Functions and Orthogonal Polynomials},
Cambridge Studies in Advanced Mathematics	(No. 153), Cambridge University Press, 2016.
		
\bibitem{blanca}  {\textsc B. Bujanda, J.L. L\'{o}pez and P.J. Pagola }, Convergent expansions of the incomplete gamma functions in terms of elementary functions, {\it Anal. Appl}, {\textbf 16} n. 3 (2018) 435-448.

\bibitem{confluent}  {\textsc B. Bujanda, J.L. L\'{o}pez and P.J. Pagola }, Convergent expansions of the confluent hypergeometric functions in terms of elementary functions. {\it Math. Comput.}, 2018.  DOI:10.1090/mcom/3389.
		
\bibitem{buhring}W.\:B\"{u}hring, Generalized Hypergeometric functions at unit argument, {\it Proc. Am. Math. Soc.}, {\textbf  114} n. 1 (1992), 145--153.

\bibitem{BusIsm}J.\:Bustoz and M.E.H.\:Ismail, On gamma function inequalities, {\it Math. Comp.} 47 (1986), 659--667.

\bibitem{FO} P.\:Flajolet And A.\:Odlyzko, Singularity Analysis Of Generating Functions, {\it SIAM J. Disc. Math.} Vol. 3, No. 2 (1990), 216--240.
		
\bibitem{dimitri}  D.B.\:Karp, Representations and inequalities for generalized hypergeometric functions, {\it J. Math. Sci.}, {\textbf  6} (2015), 885--897.

\bibitem{KLJAT2017}D.\:Karp and J.L.\:L\'{o}pez, Representations of hypergeometric functions for arbitrary values of the parameters and their use, {\it J. Approx. Theor.}, Volume 218 (2017), 42--70.

\bibitem{KL2018IJAM}D.\:Karp and J.L.\:L\'{o}pez,  On a particular class of Meijer's G functions appearing in fractional calculus, International Journal of Applied Mathematics,
Vol. 31, No.5 (2018), 521--543.

\bibitem{KPSIGMA} D.\:Karp and E.\:Prilepkina, Hypergeometric differential equation and new identities for the coefficients of
N{\o}rlund and B\"{u}hring, {\it SIGMA} 12 (2016), 052, 23 pages.

\bibitem{KilSaig}A.A.\,Kilbas, M.\,Saigo, H-transforms and applications, \textit{Analytical Methods and Special Functions}, Volume 9,
Chapman \& Hall/CRC, 2004.

\bibitem{KiryakovaBook} V.S.\,Kiryakova, Generalized Fractional Calculus and Applications, Pitman Research Notes in. Math. Series No. 301, Longman Group UK Ltd., 1994.

\bibitem{Lopez} J.L.\:L\'{o}pez, Convergent expansions of the Bessel functions in terms of elementary functions, {\it Adv. Comput. Math.}, 2018, Volume 44,
Issue 1, 277--294.

\bibitem{LukeBook} Y.L.\:Luke, \textit{The special functions and their approximations}. Volume 1. Academic Press, 1969.

\bibitem{NIST}F.W.J.\,Olver, D.W.\,Lozier,  R.F.\,Boisvert C.W.\,Clark (Eds.) \textit{NIST Handbook of Mathematical Functions}, Cambridge University Press, 2010.

\bibitem{Norlund} N.E.\,N{\o}rlund, Hypergeometric functions, {\it Acta Mathematica}, volume 94 (1955), 289--349.

\bibitem{PBM2}  A.P.\,Prudnikov,  Yu.A.\,Brychkov  and  O.I.\,Marichev, \textit{Integrals and series}, Volume 2: Special Functions, Gordon and Breach Science Publishers, 1990.

\bibitem{PBM3} A.P.\:Prudnikov, Yu.A.\:Brychkov, O.I.\:Marichev, \textit{Integrals and Series: More Special Functions},  Volume~3, Gordon and Breach Science Publishers, New York, 1990.
		
	\end{thebibliography}
\end{document}